\documentclass[11pt]{amsart}
\usepackage{amssymb}
\usepackage{xcolor}
\usepackage{mathrsfs}
\numberwithin{equation}{section}
\allowdisplaybreaks[4]
\usepackage{amsfonts}
\usepackage{amsmath}
\newtheorem{Theorem}{Theorem}[section]
\newtheorem{Proposition}[Theorem]{Proposition}

\newtheorem{Lemma}[Theorem]{Lemma}

\newtheorem{Remark}[Theorem]{Remark}

\def\Xint#1{\mathchoice
{\XXint\displaystyle\textstyle{#1}}%
{\XXint\textstyle\scriptstyle{#1}}%
{\XXint\scriptstyle\scriptscriptstyle{#1}}%
{\XXint\scriptscriptstyle\scriptscriptstyle{#1}}%
\!\int}
\def\XXint#1#2#3{{\setbox0=\hbox{$#1{#2#3}{\int}$}
\vcenter{\hbox{$#2#3$}}\kern-.5\wd0}}

\def\dashint{\Xint-}

\def\bbZ{\mathbb{Z}}
\def\bbR{\mathbb{R}}

\begin{document}

\title[A New $A_p$-$A_\infty$ Estimate]{A new $A_p$-$A_\infty$ estimate for Calder\'on-Zygmund operators in spaces of homogeneous type}

\thanks{This work was partially supported  by the
National Natural Science Foundation of China(11371200), the Research Fund for the Doctoral Program
of Higher Education (20120031110023) and the Ph.D. Candidate Research Innovation Fund of Nankai University.}

\author{Kangwei Li}

\address{School of Mathematical Sciences and LPMC,  Nankai University, Tianjin~300071, China}
\email{likangwei9@mail.nankai.edu.cn}


\begin{abstract}
In this note, we study the $A_p$-$A_\infty$ estimate for Calder\'on-Zygmund operators in terms of the weak $A_\infty$ characteristics in spaces of homogeneous type. The weak $A_\infty$ class was introduced recently by Anderson, Hyt\"onen and Tapiola. Our estimate is new even in the Euclidean space.
\end{abstract}

\maketitle

\section{Introduction and Main Results}
Let $T$ be a Calder\'on-Zygmund operator and $(w,\sigma)$ be a pair of weights. In the Euclidean setting, Hyt\"onen and Lacey \cite{HL} proved that if
\[
[w,\sigma]_{A_p}:=\sup_{Q: \textup{cubes in $\bbR^n$}}\frac{w(Q)}{|Q|}\Big(\frac{\sigma(Q)}{|Q|}\Big)^{p-1}<\infty
\]
and $w,\sigma\in A_\infty$,
then the following estimate holds
\begin{equation}\label{eq:HL}
\|T(\cdot\sigma)\|_{L^p(\sigma)\rightarrow L^p(w)}\le C_{n,p,T} [w,\sigma]_{A_p}^{\frac 1p}([w]_{A_\infty}^{\frac 1{p'}}+
[\sigma]_{A_\infty}^{\frac 1p}).
\end{equation}
It is well known that \eqref{eq:HL} extends the $A_2$ theorem, which was first proved by Hyt\"onen \cite{H}, see also in \cite{L} for a simple proof by Lerner, and in \cite{AV}, Anderson and Vagharshakyan also gave a proof in the spaces of homogeneous type. Our goal is to extend \eqref{eq:HL} with the weak $A_\infty$ characteristics (which will be introduced below) to the spaces of homogeneous type.

Now let us recall some definitions. By a space of homogeneous type ($SHT$) we mean an ordered triple $(X,\rho,\mu)$, where $X$ is a set, $\rho$ is a
quasi-metric on $X$, i.e.,
\begin{enumerate}
\item $\rho(x,y)=0$ if and only if $x=y$;
\item $\rho(x,y)=\rho(y,x)$ for all $x,y\in X$;
\item $\rho(x,z)\le \kappa(\rho(x,y)+\rho(y,z))$ for some $\kappa\ge 1$ and all $x,y,z\in X$;
\end{enumerate}
and $\mu$ is a nonnegative Borel measure on $X$ which satisfies the following doubling condition
\[
\mu(B(x,2r))\le D \mu(B(x,r)),
\]
where $B(x,r):=\{y\in X: \rho(x,y)<r\}$ and the dilation of a ball $B:=B(x,r)$ denoted by $\lambda B$ will be understood as $B(x,\lambda r)$. We point out that the doubling property implies that  any ball $B(x,r)$  can be covered by at most $N:=N_{D,\kappa}$ balls of radius $r/2$.
Next let us introduce the weak $A_\infty$ class, which was first introduced by Anderson, Hyt\"onen and Tapiola in \cite{AHT}. For every $\delta>1$, we say $w$ belongs to $\delta$-weak $A_\infty$ class $A_\infty^\delta$
if
\[
[w]_{A_\infty^\delta}:=\sup_{B}\frac 1{w(\delta B)}\int_B M(\mathbf 1_B w)(y) d\mu(y)<\infty,
\]
where the supremum is taken over all balls $B\subset X$. We collect some properties of this weak $A_\infty$ class and refer the readers to \cite{AHT} for a proof.
\begin{Proposition}
\begin{enumerate}
\item $A_\infty^{\delta}=A_\infty^{\delta'}$ for all $\delta,\delta'>\kappa$. So hereafter, we denote by $A_\infty^{\textup{weak}}:=A_\infty^{2\kappa}$ the weak $A_\infty$ class;
\item For any $w\in A_\infty^{\textup{weak}}$, we have $[w]_{A_\infty^{\textup{weak}}}\ge 1/{2(2\kappa)^{\log_2 N}}$;
\item Let $w\in A_\infty^{\rm{weak}}$. Then there exists a constant $\alpha:=\alpha(\kappa, D)$  such that for
every $0<\epsilon \le \frac 1{\alpha [w]_{A_\infty}^{\rm{weak}}}$,
\begin{equation}\label{eq:RHI}
\bigg(\dashint_B w^{1+\epsilon} d\mu\bigg)^{\frac 1{1+\epsilon}}\lesssim \dashint_{2\kappa B}w d\mu.
\end{equation}
\end{enumerate}
\end{Proposition}

Now we are ready to state the main result in this paper.
\begin{Theorem}\label{thm:m1}
Given $p$, $1<p<\infty$ and an $SHT(X,\rho, \mu)$. Let $T$ be any Calder\'on-Zygmund operator and $(w, \sigma)$ be a pair of weights. Then
we have
\[
\|T(\cdot \sigma)\|_{L^p(\sigma)\rightarrow L^p(w)}\le
C [w,\sigma]_{A_p}^{\frac 1p}(([w]_{A_\infty}^{\rm{weak}})^{\frac 1{p'}}+ ([\sigma]_{A_\infty}^{\rm{weak}})^{\frac 1{p}}),
\]
where
\[
[w,\sigma]_{A_p}:=\sup_{B: \textup{balls in $X$}}\bigg(\dashint_B wd\mu\bigg)\bigg( \dashint_{B} \sigma d\mu\bigg)^{p-1}
\]
and the constant $C$ is independent of the weights $(w, \sigma)$.
\end{Theorem}
\begin{Remark}
Note that the result is new already in the case that $X=\bbR^n$ with Euclidean distance and Lebesgue measure, since
the weak $A_\infty$ class is strictly larger than classical $A_\infty$ already in this setting.
\end{Remark}

\section{Proof of the Main result}
In this section, we will give a proof for Theorem~\ref{thm:m1}. First, we introduce the bump conditions. By a Young function $\phi$, we mean that $\phi: [0, \infty)\rightarrow [0, \infty)$ is continuous, convex
and increasing satisfying $\phi(0) = 0$ and $\phi(t)/t \rightarrow \infty$ as $t \rightarrow \infty$. Recall that the
complementary function of $\phi$, denoted by $\bar \phi$, is defined by
\[
\bar \phi(t):= \sup_{s>0}\{st- \phi(s)\}.
\]
 Given two Young functions $\Phi, \Psi$, define
\begin{eqnarray*}
[w,\sigma]_{\Phi, p}&:=&\sup_B \bigg(\dashint_B w d\mu\bigg)^{\frac 1p}\|\sigma^{\frac 1{p'}}\|_{\Phi, B},\\
{[\sigma, w]_{\Psi, p}}&:=&\sup_B \|w^{\frac 1p}\|_{\Psi, B}\bigg(\dashint_B\sigma d\mu\bigg)^{\frac 1{p'}},
\end{eqnarray*}
where the supremum is taken over all balls in $X$ and the Luxemburg norm is defined by
\[
\|f\|_{\Phi, B}:=\inf \{ \lambda>0, \dashint_B \Phi\Big(\frac f \lambda\Big) d\mu \le 1 \}.
\]
There is a famous problem named the separated bump conjecture, which states that for any Calder\'on-Zygmund operator
$T$, if
\begin{equation}\label{eq:bump}
[w,\sigma]_{\Phi, p}+[\sigma, w]_{\Psi, p}<\infty,
\end{equation}
where $\bar\Phi\in B_p$ and $\bar\Psi\in B_{p'}$ (the $B_p$ condition is recalled in \eqref{eq:bp} below),
then $T(\cdot\sigma)$ is bounded from $L^p(\sigma)$ to $L^p(w)$.
For
the so-called log-bumps, namely, when
\[
\Phi(t)=t^{p'}\log(e+t)^{p'-1+\delta}\,\,\mbox{and}\,\, \Psi(t)=t^p\log(e+t)^{p-1+\delta},
\]
this conjecture has been verified in \cite{CRV} in the Euclidean setting and for the spaces of homogeneous type, see \cite{ACM}. For more about the separated bump conjecture, see \cite{Lac, NRV} and the references therein.  In the rest of this paper, by carefully calculating the constants, we will show
the following quantitative estimate for the power bumps.
\begin{Theorem}\label{thm:m2}
Given $p$, $1<p<\infty$ and an $SHT(X,\rho,\mu)$. Let $T$ be a Calder\'on-Zygmund operator and $(w,\sigma)$ is a pair of weights satisfying \eqref{eq:bump} for
$\Phi(t)=t^{p'r}$ and $\Psi(t)=t^{ps}$, where $1<r,s<1+2(2\kappa)^{\log_2 N}/{\alpha(\kappa, D)}$. Then we have
\[
\|T(\cdot\sigma)\|_{L^p(\sigma)\rightarrow L^p(w)}\le C_{T, p, D, \kappa}([w,\sigma]_{\Phi,p} [\bar\Phi]_{B_p}^{1/p} +[\sigma,w]_{\Psi,p}[\bar\Psi]_{B_{p'}}^{1/{p'}}
),
\]
where recall that for a Young function $\phi\in B_p$,
\begin{equation}\label{eq:bp}
[\phi]_{B_p}:=\int_{1/2}^\infty \frac{\phi(t)}{t^p}\frac{dt}{t}.
\end{equation}
\end{Theorem}

Now with Theorem~\ref{thm:m2} we are ready to prove Theorem~\ref{thm:m1}.
\begin{proof}[Proof of Theorem~\ref{thm:m1}]
In fact, taking
\[
r=1+\frac 1{\alpha [\sigma]_{A_\infty}^{\rm{weak}}}.
\]
By the reverse H\"older's inequality \eqref{eq:RHI}, we have
\begin{eqnarray*}
[w,\sigma]_{\Phi,p}&=&\sup_B \bigg(\dashint_B w d\mu\bigg)^{\frac 1p}\bigg( \dashint_B \sigma^{r}d\mu\bigg)^{\frac1{p'r}}\\
&\lesssim&\sup_B \bigg(\dashint_B wd\mu\bigg)^{\frac 1p}\bigg( \dashint_{2\kappa B} \sigma d\mu\bigg)^{\frac1{p' }}\\
&\le& C_{p,D,\kappa} [w,\sigma]_{A_p}^{1/p}.
\end{eqnarray*}
For $\Phi(t)=t^{p'r}$, by definition, we know
\[
\bar \Phi(t)=t^{(p'r)'}\Big(\frac 1{p'r}\Big)^{\frac 1{p'r-1}}\frac{p'r-1}{p'r}\eqsim_{p,\kappa, D} t^{(p'r)'}.
\]
Hence
\begin{eqnarray*}
 [\bar\Phi]_{B_p} \eqsim_{p,\kappa, D}\int_{1/2}^\infty \frac {t^{(p'r)'}}{t^p}\frac {dt}t= 2^{\frac{(r-1)p}{ p'r-1}}\frac{ p'r-1}{(r-1)p}\le C_{p,\kappa, D} [\sigma]_{A_\infty}^{\rm{weak}} .
\end{eqnarray*}
And by taking similar value of $s$ we can get the result as desired.
\end{proof}
In the rest of this paper, we will focus on the proof of Theorem~\ref{thm:m2}.
We will reduce the estimates for Calder\'on-Zygmund operators to the so-called sparse operators. So first let us introduce the following result, which can be found in \cite{HK}, see also in \cite{C}. Here we follow the version used in \cite{AHT}.
\begin{Theorem}\label{dyadicsystem}
Let $0<\eta<1$ satisfy $96\kappa^6\eta\le 1$. Then there exists countable sets of points $\{z_\alpha^{k,t}: \alpha\in \mathscr A_k \}$, $k\in \bbZ$, $t=1,2,\cdots,K=K(\kappa, N,\eta)$, and a finite number of dyadic systems $\mathscr D^t:=
\{Q_\alpha^{k,t}: \alpha\in \mathscr A_k, k\in\bbZ\}$, such that
\begin{enumerate}
\item for every $t\in \{1,2,\cdots,K\}$ we have
\begin{enumerate}
\item $X=\cup_{\alpha\in \mathscr A_k}Q_\alpha^{k,t}$(disjoint union) for every $k\in\bbZ$;
\item $Q, P\in \mathscr D^t$ $\Rightarrow$ $Q\cap P=\{\emptyset, Q, P\}$;
\item $Q_\alpha^{k,t}\in \mathscr D^t$ $\Rightarrow$ $B(z_\alpha^{k,t}, c_1\eta^k)\subseteq Q_\alpha^{k,t}\subset B(z_\alpha^{k,t}, C_1\eta^k) $, where $c_1:=(12\kappa^4)^{-1}$ and $C_1:=4\kappa^2$;
\end{enumerate}
\item for every ball $B=B(x,r)$ there exists a cube $Q_B\in \cup_t \mathscr D^t$ such that $B\subseteq Q_B$ and $l(Q_B)=\eta^{k-1}$, where $k$ is the unique integer such that $\eta^{k+1}<r \le \eta^k$ and $l(Q_B)=\eta^{k-1}$ means that $Q_B=Q_\alpha^{k-1,t}$ for some indices $\alpha$ and $t$.
\end{enumerate}
\end{Theorem}
By the doubling property, we know that $\mu(B(x,r))\eqsim\mu (Q_B)$. And if $Q_\alpha^{k,t}\subset Q_\beta^{k-1,t}$, by the doubling property we also know that there exists some constant $C_{\kappa, D}$ such that
$$\mu(Q_\beta^{k-1,t})\le \mu (B(z_\beta^{k-1,t}, C_1\eta^{k-1}))\le C_{\kappa, D} \mu(B(z_\alpha^{k,t}, c_1\eta^k))\le
C_{\kappa, D}\mu(Q_\alpha^{k,t}).$$

Theorem~\ref{dyadicsystem} characterizes the structure of dyadic system in spaces of homogeneous type. See also in \cite{HK1} for an exact characterization of which kinds of sets can be dyadic cubes.
Now with Theorem~\ref{dyadicsystem} we can get the following result, which was proved in \cite{ACM}.
\begin{Lemma}\label{lm:dyadicbump}
Given a pair of weights $(w,\sigma)$, and Young functions $\Phi$ and $\Psi$.
\[
[w,\sigma]_{\Phi,p}\eqsim \max_{t\in \{1,2,\cdots, K\}}[w,\sigma]_{\Phi,p}^{\mathscr D^t},\qquad
[\sigma,w]_{\Psi,p}\eqsim   \max_{t\in \{1,2,\cdots, K\}}  [\sigma,w]_{\Psi,p}^{\mathscr D^t},
\]
where
\[
[w,\sigma]_{\Phi,p}^{\mathscr D^t}:=\sup_{Q\in\mathscr D^t} \bigg(\dashint_Q w d\mu\bigg)^{\frac 1p}\|\sigma^{\frac 1{p'}}\|_{\Phi, Q}
\]
and $[\sigma,w]_{\Psi,p}^{\mathscr D^t}$ is defined similarly.
\end{Lemma}

Now for any fixed $\mathscr D^t$, $t\in \{1,2,\cdots, K\}$, we call a family $\mathcal S\subset \mathscr D^t$ sparse
if for any $Q\in \mathcal S$, $\mu(E(Q))\ge \frac 12 \mu(Q) $, where $E(Q)=Q\setminus \cup_{Q'\in\mathcal S, Q'\subsetneq Q} Q'$.

Our purpose is to reduce the estimates for Calder\'on-Zygmund operators to the following so-called sparse operators,
\[
T^{\mathcal S} (f)(x):=\sum_{Q\in \mathcal S} \dashint_Q f(y)d\mu(y) \mathbf 1_Q(x),
\]
where $\mathcal S\subset \mathscr D$ is a sparse family in some dyadic system $\mathscr D$.
In \cite{L}, Lerner gave a nice formula which reduces the norm of Calder\'on-Zygmund operators to such kind of sparse operators. (In the recent book by Lerner and Nazarov \cite{LN}, it has been shown that Calder\'on-Zygmund operators can be dominated pointwise by the sparse operators.)  In \cite{ACM}, the authors showed that Lerner's formula also holds in spaces of homogeneous type.
\begin{Lemma}
Given an $SHT(X, \rho, \mu)$ and a Calder\'on-Zygmund operator $T$, then
for any Banach function space $Y$,
\[
\|T(f\sigma)\|_{Y}\le C_{D,\kappa}\sup_{\mathscr D^t, \mathcal S} \|T^{\mathcal S}(f\sigma)\|_{Y},
\]
where the supremum is taken over every dyadic system $\mathscr D^t$, $t=1,2,\cdots, K$ and every sparse family $\mathcal S$ in $\mathscr D^t$.
\end{Lemma}

In the rest of this paper, we only need to prove Theorem~\ref{thm:m2} for sparse operators.  We follow the strategy of \cite{CRV}. We further reduce the problem to estimate testing condition. To be precise, we have the following, see \cite{LSU} for a proof.
\begin{Lemma}\label{lm:testing}
For fixed $t\in\{1,2,\cdots,K\}$, suppose $\mathcal S$ is a sparse family in $\mathscr D^t$. Then
\begin{eqnarray}
\|T^{\mathcal S}(\cdot \sigma)\|_{L^p(\sigma)\rightarrow L^p(w)}
&\le& \sup_R \frac{\| \sum_{Q\in \mathcal S\atop Q\subset R} \dashint_Q \sigma d\mu \mathbf 1_Q(x) \|_{L^p(w)}}{\sigma(R)^{1/p}}\label{eq:testing}\\
&&\quad
+\sup_R\frac{\| \sum_{Q\in \mathcal S\atop Q\subset R} \dashint_Q w d\mu \mathbf 1_Q(x)  \|_{L^{p'}(\sigma)}}{w(R)^{1/{p'}}}\nonumber.
\end{eqnarray}
\end{Lemma}
Before we give further estimates, we introduce the following result. In the Euclidean case this is due to P\'erez \cite{P}, and in spaces of homogeneous type, see  P\'erez and Wheeden \cite{PW} and Pradolini and Salinas \cite{PS}. We give the version used in \cite{ACM}.
\begin{Lemma}\label{lm:maximalbump}
Given $p$, $1<p<\infty$ and an $SHT(X,\rho,\mu)$ and a Young function $\Phi$ such that $\Phi\in B_p$, then
\[
\|M_\Phi^{\mathscr D}\|_{L^p(\mu)}\le C_{\kappa, D}[\Phi]_{B_p}^{1/p}\|f\|_{L^p(\mu)},
\]
where $\mathscr D$ is some dyadic system in $X$ and
\[
M_\Phi^{\mathscr D} f(x):=\sup_{Q\ni x\atop Q\in \mathscr D}\|f\|_{\Phi, Q}.
\]
\end{Lemma}
Then by Theorem~\ref{dyadicsystem} we immediately get
\begin{equation}\label{eq:maximal}
\|M_\Phi\|_{L^p(\mu)}\le C_{\kappa, D}'[\Phi]_{B_p}^{1/p}\|f\|_{L^p(\mu)}.
\end{equation}

Now by symmetry we concentrate on the first term of \eqref{eq:testing}. We follow  the technique introduced in \cite{HL}, see also in \cite{CRV}.
For convenience, set $\langle f\rangle_Q= \dashint_Q f d\mu$ and denote
\[
\mathcal S_a:= \{ Q\in \mathcal S: 2^a<  \langle w\rangle_Q  \langle\sigma \rangle_Q ^{p-1} \le 2^{a+1} \,\,\mbox{and}\,\, Q\subset R \}.
\]
Denote by $\mathcal P_0^a$
the collection of maximal cubes in $\mathcal S_a$. Now we define
\[
\mathcal P_n^a:=\{\mbox{maximal cubes}\,\, P'\subset P \in \mathcal P_{n-1}^a \,\,\mbox{such that}\,\,
P'\in \mathcal S_a, \langle \sigma\rangle_{P'} > 2\langle \sigma\rangle_{P} \}.
\]
Then denote $\mathcal P^a:=\cup_n \mathcal P_n^a$. For any $P\in \mathcal P^a$, set
\[
\mathcal S_a(P):=\{Q\in \mathcal S_a : \pi(Q)=P\},
\]
where $\pi(Q)$ is the minimal cube in $\mathcal P^a$ which contains $Q$.
 We have the following Lemma, which was proved in \cite{HL} (see also in \cite{LM}) in the Euclidean setting and it is still valid for the spaces of homogeneous type with no change of the proof.
\begin{Lemma}\label{decay}
There exists a constant $c$ such that
\[
w\{x\in P: T^{\mathcal S_a(P)}(\sigma)>t  \langle \sigma\rangle_P\}\lesssim e^{-ct}w(P),
\]
where
\[
T^{\mathcal S_a(P)} (f)(x):=\sum_{Q\in \mathcal S_a(P)} \dashint_Q f(y)d\mu(y) \mathbf 1_Q(x).
\]
\end{Lemma}
Now we are ready to estimate the first term on the right side of \eqref{eq:testing}.
We have
\begin{eqnarray*}
\Big\| \sum_{Q\in \mathcal S\atop Q\subset R}\langle\sigma\rangle_Q \mathbf 1_Q(x) \Big\|_{L^p(w)}
&\le&\sum_{a}\Big\| \sum_{Q\in \mathcal S_a } \langle\sigma\rangle_Q \mathbf 1_Q(x) \Big\|_{L^p(w)} \\
&=&\sum_{a}\Big\|\sum_{P\in \mathcal P^a}T^{\mathcal S_a(P)}(\sigma) \Big\|_{L^p(w)}.
\end{eqnarray*}
Set
\[
L_{j}^a(P):=\bigg\{x: T^{\mathcal S_a(P)}(\sigma)(x)\in [j, j+1)\langle\sigma\rangle_P\bigg\}.
\]
By Lemma~\ref{decay} we have
\begin{eqnarray*}
\Big\|\sum_{P\in \mathcal P^a} T^{\mathcal S_a(P)}(\sigma) \Big\|_{L^p(w)}
&\le&\sum_{j=0}^\infty(j+1)\Big\|\sum_{P\in \mathcal P^a} \langle\sigma\rangle_P\mathbf 1_{L_j^a(P)}(x) \Big\|_{L^p(w)}\\
&\lesssim&\sum_{j=0}^\infty(j+1)\Big(\sum_{P\in\mathcal P^a}\langle\sigma\rangle_P^p e^{-cj}w(P)\Big)^{\frac1p}\\
&\lesssim& \Big(\sum_{P\in\mathcal P^a}\langle\sigma\rangle_P^p w(P)\Big)^{\frac1p}.
\end{eqnarray*}
Therefore,
\[
\Big\| \sum_{Q\in \mathcal S\atop Q\subset R}\langle\sigma\rangle_Q \mathbf 1_Q(x) \Big\|_{L^p(w)}
\lesssim\sum_a\Big(\sum_{P\in\mathcal P^a}\langle\sigma\rangle_P^p w(P)\Big)^{\frac1p}.
\]
We follow the idea of \cite{CRV}. Define $\Phi_0(t)=t^{p'(r+1)/2}$, set $\gamma=\frac1{2(r+1)}$.
Since $r>1$ and it is dominated by some constant that depends only on the structure constant of $X$, it is easy to check that
\begin{eqnarray*}
[\bar\Phi_0]_{B_p}&=&\int_{1/2}^\infty \frac{\bar \Phi_0(t)}{t^p}\frac {dt}t\\
&\le&c_{p,\kappa, D}\frac{ p'(r+1) -2}{p(r-1)}2^{\frac{p(r-1)}{ p'(r+1) -2}}\\
&\le& c_{p,\kappa, D}' \frac{ p'r-1}{(r-1)p}2^{\frac{(r-1)p}{ p'r-1}}=c_{p,\kappa, D}''[\bar\Phi]_{B_p}.
\end{eqnarray*}
Now notice that
\[
\frac {\frac{2r+1}{4}}{r}+\frac 14=\frac34+\frac1{4r}<1.
\]
We have
\begin{eqnarray*}
\dashint_Q \sigma^{\frac {r+1}2}\le \bigg(\dashint_Q \sigma^{\frac {2r+1}4\cdot \frac r{\frac{2r+1}{4}}}\bigg)^{\frac {\frac{2r+1}{4}}{r}}\cdot  \bigg( \dashint_Q   \sigma^{\frac 14\cdot 4}     \bigg)^{\frac 14}.
\end{eqnarray*}
It follows that
\[
\|\sigma^{\frac 1{p'}}\|_{\Phi_0,Q}\le \|\sigma^{\frac 1{p'}}\|_{\Phi,Q}^{1-\gamma}\cdot\|\sigma^{\frac 1{p'}}\|_{p',Q}^{\gamma}.
\]
Therefore, by the sparseness and Lemma~\ref{lm:dyadicbump}
\begin{eqnarray*}
\sum_{P\in\mathcal P^a}\langle\sigma\rangle_P^p w(P)
&\le& \sum_{P\in\mathcal P^a} w(P)\|\sigma^{\frac 1 {p'}}\|_{\Phi_0,P}^p \|\sigma^{\frac 1p}\|_{\bar\Phi_0, P}^p\\
&\le& \sum_{P\in\mathcal P^a}\langle w\rangle_P\|\sigma^{\frac 1 {p'}}\|_{\Phi,P}^{p(1-\gamma)} \|\sigma^{\frac 1{p'}}\|_{p',P}^{p\gamma}\|\sigma^{\frac 1p}\|_{\bar\Phi_0, P}^p \mu(P)\\
&\le& [w,\sigma]_{\Phi,p}^{p(1-\gamma)}2^{(a+1)\gamma} \sum_{P\in\mathcal P^a}\|\sigma^{\frac 1p}\|_{\bar\Phi_0, P}^p \mu(P)\\
&\lesssim& [w,\sigma]_{\Phi,p}^{p(1-\gamma)}2^{(a+1)\gamma} \int M_{\bar\Phi_0}^{\mathscr D}(\mathbf 1_R\sigma^{\frac 1p})(x)^pd\mu\\
&\lesssim& [w,\sigma]_{\Phi,p}^{p(1-\gamma)}2^{(a+1)\gamma} [\bar \Phi_0]_{B_p}\sigma(R)\quad\mbox{(by Lemma~\ref{lm:maximalbump})}\\
&\lesssim&[w,\sigma]_{\Phi,p}^{p(1-\gamma)}2^{(a+1)\gamma}[\bar\Phi]_{B_p}\sigma(R).
\end{eqnarray*}
Consequently,
\begin{eqnarray*}
\frac{\Big\| \sum_{Q\in \mathcal S\atop Q\subset R}\langle\sigma\rangle_Q \mathbf 1_Q(x) \Big\|_{L^p(w)}}{\sigma(R)^{1/p}}
&\lesssim& [w,\sigma]_{\Phi,p}^{(1-\gamma)}[\bar\Phi]_{B_p}^{\frac 1p}\sum_a 2^{(a+1)\gamma/p}\\
&\le&C_{p, D, \kappa}[w,\sigma]_{\Phi,p}^{(1-\gamma)}[\bar\Phi]_{B_p}^{\frac 1p} [w,\sigma]_{A_p}^{\gamma/p}\\
&\le& C_{p, D, \kappa}[w,\sigma]_{\Phi,p} [\bar\Phi]_{B_p}^{\frac 1p} .
\end{eqnarray*}
This completes the proof.

\textbf{Acknowledgements}.\,\, This work was done while the author was visiting Department of Mathematics and Statistics,
University of Helsinki. He thanks the Department of Mathematics and Statistics, University of Helsinki and Professor Tuomas P. Hyt\"onen for hospitality and support. He thanks Olli Tapiola for the nice talk on the weak $A_\infty$ class.
Particular thanks go to Professor Tuomas P. Hyt\"onen for carefully reading the paper and many helpful suggestions.

\end{document}